\documentclass[a4paper, reqno, twoside]{amsart}
\usepackage{amsfonts,amssymb,amsthm,amsmath}
\usepackage[all]{xy}
\usepackage{enumerate}
\usepackage[latin1]{inputenc}
\usepackage{a4wide} 
\usepackage{hyperref}
\usepackage[normalem]{ulem}
\usepackage{color}
\usepackage[mathscr]{euscript} 
\usepackage{lmodern}
\usepackage{version}
\usepackage{upgreek} 
\newtheorem{theorem}{\sc Theorem}[section]
\newtheorem{proposition}[theorem]{\sc Proposition}

\newtheorem{lemma}[theorem]{\sc Lemma}

\theoremstyle{definition}

\newtheorem{example}[theorem]{\sc Example}

\theoremstyle{remark}

\newtheorem{remark}[theorem]{\sc Remark}

\IfFileExists{tcilatex.tex}{\input{tcilatex}}{}

\newcommand{\tensor}[1]{\otimes_{\scriptscriptstyle{#1}}}

\newcommand{\fk}[1]{\mathfrak{#1}}

\renewcommand{\hom}[3]{\mathrm{Hom}_{\Sscript{#1}}\left(#2,\,#3\right)}

\newcommand{\bcirc}{{\boldsymbol{\circ}}}

\newcommand{\prlimit}[2]{\varprojlim_{#1}\left({#2}\right)}

\newcommand{\m}{\mathfrak{m}}

\renewcommand{\ker}[1]{\mathrm{Ker}\left({#1}\right)}
\newcommand{\what}[1]{\widehat{#1}}

\definecolor{bostonuniversityred}{rgb}{0.8, 0.0, 0.0}

\newcommand{\N}{\mathbb{N}}
\newcommand{\K}{\Bbbk}
\newcommand{\C}{\mathbb{C}}

\newcommand{\hH}{\mathscr{H}}

\newcommand{\tT}{\mathscr{T}}

\newcommand{\cI}{{\mathcal I}}
\newcommand{\cJ}{{\mathcal J}}

\newcommand{\cL}{{\mathcal L}}

\newcommand{\Sscript}[1]{\scriptscriptstyle{#1}}

\newcommand{\limn}{\underset{n\to\infty}{\lim}}

\newcommand{\Sqx}{(x_{n})_{\Sscript{n\, \geq \, 0}}}
\newcommand{\Sqy}{(y_{n})_{\Sscript{n\, \geq \, 0}}}
\newcommand{\Sqz}{(z_{n})_{\Sscript{n\, \geq \, 0}}}
\newcommand{\Squ}{(u_{n})_{\Sscript{n\, \geq \, 0}}}

\newcommand{\Sqa}{(a_{n})_{\Sscript{n\, \geq \, 0}}}

\excludeversion{invisible*}

\begin{document}
\allowdisplaybreaks

\title[Comparing topologies on linearly recursive sequences]{Comparing topologies on linearly recursive sequences.}

\author{Laiachi El Kaoutit}
\address{Universidad de Granada, Departamento de \'{A}lgebra and IEMath-Granada. Facultad de Educaci\'{o}n, Econon\'ia y Tecnolog\'ia de Ceuta. Cortadura del Valle, s/n. E-51001 Ceuta, Spain}
\email{kaoutit@ugr.es}
\urladdr{http://www.ugr.es/~kaoutit/}

\author{Paolo Saracco}
\address{University of Turin, Department of Mathematics ``Giuseppe Peano'', via Carlo Alberto 10, I-10123 Torino, Italy}
\email{p.saracco@unito.it}
\urladdr{sites.google.com/site/paolosaracco}

\date{\today}
\thanks{This paper was written while P. Saracco was member of the ``National Group for Algebraic and Geometric Structures, and their Applications'' (GNSAGA - INdAM). His stay as visiting researcher at the campus of Ceuta of the University of Granada was financially supported by IEMath-GR. Research supported by grant PRX16/00108, and Spanish Ministerio de Econom\'{\i}a y Competitividad  and the European Union FEDER, grant MTM2016-77033-P}

\begin{abstract}
The space of linearly recursive sequences of complex numbers admits two distinguished topologies. Namely, the adic topology induced by the ideal of those sequences whose first term is $0$ and the topology induced from the Krull topology on the space of complex power series via a suitable embedding. We show that these topologies are not equivalent.
\end{abstract}

\maketitle


\pagestyle{headings}

\section{Introduction}

A linearly recursive sequence of complex numbers is a sequence of elements of $\C$ which satisfies a recurrence relation with constant coefficients.  These sequences arise widely in mathematics and have been studied extensively and from different points of view, see \cite{poorten}  for a survey on the subject. 
Classically they are related with formal power series, in the sense that a sequence $(s_n)_{\Sscript{n \geq 0}}$ is linearly recursive if and only if its generating function $\sum_{\Sscript{n \geq 0}}s_nZ^n$ is a rational function $p(Z)/q(Z)$, where $p(Z),q(Z)\in\C[Z]$ and  $q(0)\neq 0$. Nevertheless, few topological properties seems to be known. For instance, it is known that the space $\cL in(\C)$ of all linearly recursive sequences of any order forms an augmented algebra under the Hurwitz product, with augmentation given by the projection on the $0$-th component. As such, it comes endowed with a natural topology, which is the adic topology induced by the kernel $J$ of this augmentation.  Besides, there is a monomorphism of algebras  which assigns every linearly recursive sequence $(s_n)_{\Sscript{n \geq 0}}$ to the power series $\sum_{\Sscript{n \geq 0}}\left({s_n}/{n!}\right)Z^n$. Through this monomorphism, the algebra of  linearly recursive sequences can be considered as a subalgebra of $\C[[Z]]$ and, as such, it inherits another natural topology, namely the one induced by the Krull topology on $\C[[Z]]$ (the adic topology induced by the unique maximal ideal $\m$ of $\C[[Z]]$, which is also the augmentation ideal induced by the evaluation at $0$).

These two topologies are very close. Namely, up to the embedding above, one can see that $J=\m\cap \cL in(\C)$, so that the adic topology is finer than the induced one. A natural question which arises is if these are equivalent or not.  Notice that, since the finiteness hypotheses are not fulfilled, Artin-Rees Lemma fails to be applied in this context. 
In fact, in this note we will compare these two topologies and we will give a negative answer to the previous question: the $J$-adic topology is strictly finer than the induced one. 

As a by-product, we will see that the adic completion of $\cL in(\C)$ is larger than its completion with respect to the induced topology, which in fact can be identified with $\C[[Z]]$ itself, in the sense that we will provide a split surjective morphism $\what{\cL in(\C)}\to \C[[Z]]$.

Our approach will take advantage of the fact that, from an algebraic point of view, linearly recursive sequences may be identified with the finite (or continuous) dual of the algebra   $\C[X]$ of polynomial functions of the additive affine algebraic $\C$-group, and that the formal power series can be considered as its full linear dual $\C[X]^*$.

The main motivation behind this note comes from studying the completion of the finite dual Hopf algebra of the universal enveloping algebra of a finite-dimensional complex Lie algebra.

\section{The space of linearly recursive sequences and Hurwitz's product}\label{sec:linrecursion}

We assume to work over the field $\C$ of complex numbers. However, it will be clear that this choice is not restrictive as the results will hold as well if we consider any algebraically closed field $\K$ of characteristic $0$ instead. An augmented complex algebra $A$, is an algebra endowed with a morphism of algebras $ A \to \C$, called the \emph{augmentation}. 

All vector spaces, algebras and coalgebras are assumed to be over $\C$. The unadorned tensor product $\otimes$ denotes the tensor product over $\C$. All maps are assumed to be at least $\C$-linear. For every vector space $V$, the $\C$-linear dual of $V$ is $V^*=\hom{\C}{V}{\C}$ (i.e., the vector space of all linear forms from $V$ to $\C$). Given a coalgebra $C$, for dealing with the comultiplication of an element $x\in C$ we will resort to the Sweedler's Sigma notation $\Delta(x)=\sum x_1\otimes x_2$.

Consider the vector space $\C^\N$ of all sequences $\Sqz$ of complex numbers. A sequence $\Sqz\in\C^\N$ is said to be \emph{linearly recursive} if there exists a family of constant coefficients $c_1,\ldots,c_r\in\C$, $r\geq 1$, such that
\begin{equation*}
z_n = c_1 z_{n-1} + c_2 z_{n-2} + \cdots + c_r z_{n-r} \qquad \text{for all } n\geq r.
\end{equation*}
Denote by $\cL in(\C)\subseteq \C^\N$ the vector subspace of all linearly recursive sequences.

Then the study of the algebraic and/or topological properties of the vector space $\cL in(\C)$ depends heavily on which product we are choosing on the vector space $\C^{\N}$, since the latter can be endowed with at least two algebra structures, as the subsequent Lemma \ref{lema:Fornaca} shows. 


\begin{lemma}\label{lema:Fornaca}
The assignment $\Upphi:\C^\N\to \C[X]^*$ given by
\begin{equation*}
\big[\Upphi\big(\Sqz\big)\big](X^m)=z_m \qquad \text{for all }m\geq 0
\end{equation*}
is an isomorphism of vector spaces.
\end{lemma}

Next we recall how the vector space $\cL in(\C)$ can be endowed with a Hopf algebra structure, by using the Hurwitz's product. Recall first that  $\C[X]$ is in fact a Hopf algebra, as it can be identified with the algebra of polynomial functions on the affine complex line $\mathbb{A}_\C^1=\C$, viewed as an algebraic group with the sum. Comultiplication, counit and antipode are the algebra morphisms induced by the assignments
\begin{equation*}
\Delta(X)=X\otimes 1+1\otimes X, \qquad \varepsilon(X)=0,\qquad S(X)=-X.
\end{equation*}
From this it follows that $\C[X]^*$ is an augmented algebra under the \emph{convolution product}
\begin{equation}\label{Eq:convolution}
(f*g)\left(X^n\right)=\sum_{k=0}^n\binom{n}{k}f\left(X^k\right)g\left(X^{n-k}\right)\qquad \text{for all } n\geq 0.
\end{equation}
The unit of $\C[X]^*$ is the counit $\varepsilon$ of $\C[X]$. The augmentation $\varepsilon_*$ is given by evaluation at $1$.

As a consequence, the vector space $\C^\N$ turns out to be an augmented algebra as well in such a way that $\Upphi$ becomes an algebra isomorphism. The product of this algebra is the so-called \emph{Hurwitz's product}
\begin{equation}\label{Eq:Hproduct}
\Sqz \,*\,\Squ=\left(\sum_{k=0}^n\binom{n}{k}z_k\, u_{n-k}\right)_{\Sscript{n\,\geq \,0}}.
\end{equation}
The unit is the sequence $\Sqz$ with $z_0=1$ and $z_n=0$ for all $n\geq 1$. The augmentation is given by the projection on the $0$-th component. The vector space of  all sequences $\C^{\N}$ endowed with this algebra structure will be denoted by $\hH\C^{\N}$.

Recall also that given a Hopf algebra as $\C[X]$, we may consider its \emph{finite dual }\footnote{In the literature, it appears also under the names \emph{Sweedler dual} or \emph{continuous dual}, where continuity is with respect to the linear topology whose neighbourhood base at $0$ consists exactly of the finite-codimensional ideals,  see e.g.~\cite[\S3]{petersontaft}.} Hopf algebra $\C[X]^\bcirc$. This is the vector subspace of all linear maps which vanish on a finite-codimensional ideal (i.e., one that leads to a finite-dimensional quotient algebra). Here we will focus only on the case of our interest and we refer to \cite[Chapter 9]{Montgomery:1993} and \cite[Chapter VI]{sweedler} for a more extended treatment.

\begin{lemma}
Given the Hopf algebra $\C[X]$, the set
\begin{equation*}
\C[X]^\bcirc=\big\{f\in\C[X]^*\mid \ker{f}\supseteq I,\text{ for } I \text{ a non-zero ideal of }\C[X]\big\}
\end{equation*}
is an augmented subalgebra of $\C[X]^*$ which is also a Hopf algebra. The augmentation $\varepsilon_\bcirc$ is given by the restriction of $\varepsilon_*$. The comultiplication on $\C[X]^\bcirc$ is defined in such a way that for $f\in\C[X]^\bcirc$, we have
\begin{equation}\label{eq:comultiplication}
\Delta_\bcirc(f)=\sum f_1\otimes f_2 \; \iff \; \Big( f(pq)=\sum f_1(p)f_2(q),\, \text{ for all }\, p,q\in \C[X]\Big).
\end{equation}
The antipode is given by pre-composing with the one of $\C[X]$, i.e.,~$S_\bcirc(f)=f\circ S$ for all $f\in\C[X]^\bcirc$.
\end{lemma}

Since the algebra $\C[X]$ is a principal ideal domain, it turns out that the space of linearly recursive sequences $\cL in(\C)$ can be identified with  $\C[X]^\bcirc$ via the isomorphism $\Upphi$, whence it becomes an augmented subalgebra of $\hH\C^\N$ and a Hopf algebra.
For further reading, we refer the interested reader to \cite{petersontaft,futiamullertaft,taft} and \cite[Chapter 2]{underwood}.

\section{Two filtrations on the space of linearly recursive sequences}\label{sec:equivfiltrations}

In this and in the next section, we will implicitly make use of the identifications $\hH\C^\N=\C[X]^*$ and $\cL in(\C)=\C[X]^\bcirc$, via the isomorphism of algebras $\Upphi$ of Lemma \ref{lema:Fornaca}. 

As augmented algebras, both $\C[X]^*$ and $\C[X]^\bcirc$ inherit a natural filtration. Namely, if we let $I:=\ker{\varepsilon_*}$ and $J:=\ker{\varepsilon_\bcirc}$ be their \emph{augmentation ideals}, then we can consider $\C[X]^*$ and $\C[X]^\bcirc$ as filtered with the adic filtrations $F_n\left(\C[X]^*\right) = I^n$ and $F_n\left(\C[X]^\bcirc\right)=J^n$, $n\geq 0$.

Moreover, $\C[X]^\bcirc$ inherits a filtration $F'_{n}\left(\C[X]^\bcirc\right) = I^n\cap\C[X]^\bcirc$ induced from the canonical inclusion $\C[X]^\bcirc\subseteq \C[X]^*$ as well and it is clear that $F_n\left(\C[X]^\bcirc\right)\subseteq F_n'\left(\C[X]^\bcirc\right)$. Hence,  the $J$-adic filtration on $\C[X]^\bcirc$ is finer than the induced one.  As we will show in this section, it is in fact strictly finer. 

For every $\lambda\in\C$ we define $\phi_\lambda:\C[X]\to \C$ to be the algebra map such that $\phi_\lambda(X)=\lambda$. 
The set $G_a:=\mathsf{Alg}_\C\left(\C[X],\C\right)=\left\{\phi_\lambda\mid \lambda\in\C\right\}$ is a group with group structure given by
$$
\phi_{\lambda} \cdot \phi_{\lambda'} \,:=\, \phi_{\lambda} * \phi_{\lambda'} \,=\, \phi_{\lambda \, +\, \lambda'}, \qquad e_{G_a}\,:=\, \varepsilon\,=\,\phi_{0}, \qquad \left(\phi_\lambda\right)^{-1} \, :=\, \phi_\lambda\circ S \,=\,\phi_{-\lambda} .
$$

\begin{lemma}[{\cite[Example 9.1.7]{Montgomery:1993}}]\label{lemma:CGKMM}
Denote by $\xi$ the distinguished element in $\C[X]^*$ which satisfies $\xi(X^n)=\delta_{n,1}$ for all $n\geq 0$ (Kronecker's delta). Then the convolution product \eqref{Eq:convolution} induces an isomorphism of Hopf algebras 
\begin{equation}\label{eq:CGKMM}
\Uppsi:\C[\xi]\otimes \C G_a {\longrightarrow} \C[X]^\bcirc, \quad \Big( \xi^{n}\tensor{}\phi_{\lambda} \longmapsto \xi^{n} \, *\, \phi_{\lambda}\Big), 
\end{equation}
where $\C G_a$ is the group algebra on $G_a$ and $\C[\xi]$ is a polynomial Hopf algebra as in \S\ref{sec:linrecursion}.
\end{lemma}
\noindent We denote by 
\begin{equation}\label{Eq:SanSecondo}
\xymatrix@C=35pt{\vartheta: \C[\xi] \,  \ar@<-0.2ex>@{^{(}->}^-{\uppsi}[r]  & \ar@<-0.2ex>@{^{(}->}^-{\iota}[r] \C[X]^\bcirc  \,  & \C[X]^* }
\end{equation}
the algebra monomorphism induced by $\Uppsi$. 

\begin{remark}
It is worthy to point out that Lemma \ref{lemma:CGKMM} is a particular instance of the renowned \emph{Cartier-Gabriel-Kostant-Milnor-Moore Theorem}, which states that for a cocommutative Hopf $\Bbbk$-algebra $H$ over an algebraically closed field $\Bbbk$ of characteristic zero, the multiplication in $H$ induces an isomorphism of Hopf algebras $U\left(P\left(H\right)\right)\#\, \C G\left(H\right) \cong H$, where the left-hand side is endowed with the smash product algebra structure (see \cite[Corollary 5.6.4 and Theorem 5.6.5]{Montgomery:1993}, \cite[Theorems 8.1.5, 13.0.1 and \S13.1]{sweedler} and \cite[Theorem 15.3.4]{radford}). 
\end{remark}

Denote by $\varepsilon_a:\C G_a\to \C$ the counit of the group algebra, which acts via $\varepsilon_a(\phi_\lambda)=\phi_\lambda(1)=1$ for all $\lambda\in\C$, and by $\varepsilon_\xi:\C[\xi]\to \C$ the counit of the polynomial algebra in $\xi$ defined by $\varepsilon_\xi(\xi)=0$. These maps are in fact  the restrictions of the counit $\varepsilon_\bcirc: \C[X]^\bcirc \to \C$ to the vector subspaces of $\C[X]^\bcirc$ generated by $G_a$ and $\left\{\xi^n\mid n\geq 0\right\}$, respectively. Thus, up to the isomorphism $\Uppsi$ of equation \eqref{eq:CGKMM}, we have $\varepsilon_\bcirc\,=\, \varepsilon_\xi \tensor{} \varepsilon_a$. 

\begin{lemma}\label{lemma:isograded}
The isomorphism $\Uppsi$ of \eqref{eq:CGKMM} induces an isomorphism of vector spaces
\begin{equation*}
\C\bar{\xi}\oplus \frac{\ker{\varepsilon_a}}{\ker{\varepsilon_a}^2}\cong \frac{J}{J^2},
\end{equation*}
where $\bar{\xi}=\xi+\langle \xi^2\rangle$ in the quotient $\langle \xi\rangle/\langle \xi^2\rangle$.
\end{lemma}
\begin{proof}
First of all, as $\Uppsi$ is an isomorphism of Hopf algebras, it induces an isomorphism of vector spaces between $J/J^2$ and $\ker{\varepsilon_{\xi}\otimes \varepsilon_a}/\ker{\varepsilon_{\xi}\otimes \varepsilon_a}^2$. Set $K:=\ker{\varepsilon_{\xi}\otimes \varepsilon_a}$. The family of assignments
\begin{equation*}
\frac{\langle \xi^k\rangle}{\langle \xi^{k+1}\rangle}\otimes\frac{\ker{\varepsilon_a}^h}{\ker{\varepsilon_a}^{h+1}}\longrightarrow \frac{K^{n}}{K^{n+1}}; \quad \bigg[\left(\xi^k+\langle \xi^{k+1}\rangle\right)\otimes \left(x+\ker{\varepsilon_a}^{h+1}\right)\longmapsto \left(\xi^k\otimes x\right)+ K^{n+1}\bigg]
\end{equation*}
for $h,k\geq 0$ and $n=h+k$ induces a graded isomorphism of graded vector spaces 
$$
\mathsf{gr}(\C[\xi])\otimes \mathsf{gr}(\C G_a) \,\cong\, \mathsf{gr}\big(\C[\xi]\otimes \C G_a\big),
$$
see e.g.~\cite[Lemma VIII.2]{nastasescuoystaeyen}. In particular, the degree $1$ component of this together with $\Uppsi$ induce the stated isomorphism
\begin{equation*}
\C\bar{\xi}\oplus \frac{\ker{\varepsilon_a}}{\ker{\varepsilon_a}^2}\cong \left(\frac{\langle \xi\rangle}{\langle \xi^2\rangle}\otimes \frac{\C G_a}{\ker{\varepsilon_a}}\right)\oplus \left(\frac{\C[\xi]}{\langle \xi\rangle}\otimes\frac{\ker{\varepsilon_a}}{\ker{\varepsilon_a}^2}\right)\cong \frac{K}{K^2} \cong \frac{J}{J^2}. \qedhere
\end{equation*}
\end{proof}

The key fact is that the quotient ${\ker{\varepsilon_a}}/{\ker{\varepsilon_a}^2}$ does not vanish, as we will show in the subsequent lemma. To this aim, recall that there is an algebra isomorphism 
\begin{equation}\label{Eq:LaRadio}
\Theta:\C[X]^*\longrightarrow  \C[[Z]], \quad \Big( f \longmapsto \sum_{k\geq 0}f(e_k)Z^k \Big),
\end{equation}
where $e_k=X^k/k!$ for all $k\geq 0$. Notice that $\Theta \circ \vartheta(\xi)=Z$, where $\vartheta$ is the morphism given in  \eqref{Eq:SanSecondo}.

\begin{lemma}\label{lemma:notvanish}
The element $\phi_1-\varepsilon+\ker{\varepsilon_a}^2$ is non-zero in the quotient ${\ker{\varepsilon_a}}/{\ker{\varepsilon_a}^2}$.
\end{lemma}

\begin{proof}
Assume by contradiction that $\phi_1-\varepsilon\in \ker{\varepsilon_a}^2$. By applying  $\Uppsi$, this implies that $\phi_1-\varepsilon\in J^2$, whence $\phi_1-\varepsilon\in I^2$ in $\C[X]^*$. Since $\Theta$ induces a bijection between $I^n$ and $\langle Z^n\rangle \subseteq\C[[Z]]$ for all $n\geq 1$, claiming that $\phi_1-\varepsilon\in I^2$ in $\C[X]^*$ would imply that $\sum_{k\geq 1}Z^k/k!\in \langle Z ^{2}\rangle$, which is a contradiction. Thus,  $\phi_1-\varepsilon\notin \ker{\varepsilon_a}^2$.
\end{proof}

It follows from Lemma \ref{lemma:isograded} and Lemma \ref{lemma:notvanish} that the elements $\xi+J^2$ and $\phi_1-\varepsilon+J^2$ are linearly independent in $J/J^2$. In particular, $\phi_1-\varepsilon-\xi\notin J^2$. However, $\phi_1-\varepsilon-\xi$ as an element of $\C[X]^*$ maps $e_0=1$ and $e_1=X$ to $0$ and it maps $e_n=X^n/n!$ to $1/n!$ for all $n\geq 2$. Hence $\phi_1-\varepsilon-\xi=\xi^2*h_{(2)}\in I^2$, where for every $k\geq 0$
\begin{equation*}
h_{(k)}(e_n):=\frac{1}{(n+k)!} \quad \text{ for all }n\geq 0.
\end{equation*}
Indeed,
\begin{equation*}
(\xi^2*h_{(2)})\left(e_n\right)=\sum_{i+j=n}\xi^2(e_i)h_{(2)}(e_j)=\begin{cases}0 & n=0,1 \\ \frac{1}{n!} & n\geq 2\end{cases}
\end{equation*}
This shows that $\phi_1-\varepsilon-\xi$ is an element in $\C[X]^\bcirc\cap I^2$ but not in $J^2$, so that $J^2\subsetneq \C[X]^\bcirc\cap I^2$. Now, by induction one may see that for every $n\geq 1$ the element
\begin{equation}\label{Eq:CarloAlberto}
\phi_{1}-\left( \sum_{k=0}^{n-1}\frac{1}{k!}\xi^{k}\right) \, = \, \xi^n*h_{(n)}\, \in \, I^{n}\cap \C[X]^\bcirc
\end{equation}
does not belong to $J^{n}$, so that the two filtrations do not coincide. We point out that, under the isomorphism $\C[X]^{*}\cong \C[[Z]]$ of equation \eqref{Eq:LaRadio}, the element of equation 
\eqref{Eq:CarloAlberto} corresponds to
\begin{equation*}
\mathsf{exp}(Z)-\left( \sum_{k=0}^{n-1}\frac{1}{k!}Z^{k}\right)=Z^n\cdot\left(\sum_{k\geq 0}\frac{1}{{(n+k)}!}Z^{k}\right).
\end{equation*}

Summing up, we have shown that  the $J$-adic filtration on $\C[X]^{\bcirc}$ is strictly finer than the filtration induced from the inclusion $\iota:\C[X]^\bcirc\to\C[X]^{*}$.

\section{Comparing the two topologies on \texorpdfstring{$\C[X]^\bcirc$}{}}

Recall that a filtration on an algebra naturally induces on it a linear topology, whose neighbourhood base at $0$ is given exactly by the elements of the filtration (see for example \cite[III.49, Example 3]{bourbaki:top} or \cite[\S I, Chapter D]{nastasescuoystaeyen}).
Furthermore, given an algebra $A$ endowed with the $\m$-adic filtration associated to an ideal $\m\subseteq A$, the \emph{completion} of $A$ with respect to the linear topology induced by this filtration is, by definition, $\widehat{A}=\prlimit{n}{A/\m^n}$, i.e., the projective limit of the projective system $A/\m^n$ with the obvious projection maps $A/\m^n\twoheadrightarrow A/\m^m$ for $n\geq m$. An algebra $A$ is said to be \emph{Hausdorff and complete} if the canonical map $ A \to \what{A}$ is an isomorphism. For further details, we refer to \cite[\S II, Chapter D]{nastasescuoystaeyen}.

\begin{example}\label{ex:C[X]*}
For every $n\geq 0$, there is a linear isomorphism between ${\C[X]^*}/{I^{n+1}}$ and the linear dual of the vector subspace $\C[X]_{\leq \, n}\subseteq \C[X]$ of all polynomials of degree up to $n$. These in turn induce an isomorphism
\begin{equation}\label{eq:complete}
\what{\C[X]^{*}}\,=\,\prlimit{n}{\frac{\C[X]^*}{I^{n+1}}} \,\cong\, \C[X]^*
\end{equation}
by which we conclude that $\C[X]^*$ is complete with respect to the $I$-adic topology. 
\end{example}

\begin{remark}\label{rem:Mucholio}
Let us consider again the algebra monomorphism $\vartheta:\C[\xi]\to \C[X]^{*}$ of equation \eqref{Eq:SanSecondo}. Since $\vartheta(\xi)\in \ker{\varepsilon_{*}}=I$, we have that $\vartheta$ is a filtered morphism of filtered algebras and so we may consider its completion $\what{\vartheta}:\what{\C[\xi]}\to \what{\C[X]^{*}}\cong \C[X]^{*}$. Therefore, up to the canonical identification $\C[[\xi]] = \C[[Z]]$, the map $\what{\vartheta}$ turns out to be the inverse of $\Theta$. 
A useful consequence of this is that every element $g\in\C[X]^*$ can be written as 
\begin{equation}\label{Eq:Valentino}
g\, =\, \sum_{k\geq 0}g(e_k)\xi^k,
\end{equation}
where as before $e_k=X^k/k!$ for all $k\geq 0$. By the right-hand side of equation \eqref{Eq:Valentino}, we mean the image in $\C[X]^*$ of the element
\begin{equation*}
\left(\sum_{k=0}^{n}g(e_k)\xi^k+I^{n+1}\right)_{n\geq 0}=\lim_{n\to \infty}\left(\sum_{k=0}^{n}g(e_k)\xi^k\right)
\end{equation*}
via the isomorphism \eqref{eq:complete}. Since $\xi^i\left(e_j\right)=\delta_{i,j}$ for all $i,j\geq 0$, given any $p=\sum_{i=0}^tp_ie_i\in\C[X]$ the sequence $\left(\sum_{k=0}^{n}g(e_k)\xi^k\right)(p)$, $n\geq 0$, eventually becomes constant and it equals the element $\sum_{i=0}^tp_ig(e_i)=g(p)$.  In light of this interpretation, $I^n=\langle \xi^n\rangle$ for all $n\geq 0$, in the algebra $\C[X]^{*}$.
\end{remark}

We already know from \S \ref{sec:equivfiltrations} that the $J$-adic filtration on $\C[X]^\bcirc$ do not coincide with the one induced by the inclusion $\C[X]^\bcirc\subseteq\C[X]^{*}$. Nevertheless, the topologies they induce may still be equivalent ones (that is, the two filtrations may be equivalent). Our next aim is to show that these topologies are not even equivalent, by showing that the $J$-adic completion of $\C[X]^\bcirc$ is not homeomorphic to $\C[X]^{*}$ via the completion of the inclusion map $\iota:\C[X]^\bcirc\to \C[X]^*$.

\begin{remark}\label{rem:densityEN}
It is worthy to mention that $\C[X]^\bcirc$ is dense in $\C[X]^{*}$ with respect to the finite topology on $\C[X]^{*}$ (the one induced by the product topology on $\C^{\C[X]}$), see for instance \cite[Exercise 1.5.21]{dascalescu}. On the other hand, since for every $f\in\C[X]^*$ and for all $n\geq 0$, we have that $f+\langle \xi^n\rangle = \mathcal{O}\left(f;e_0,e_1,\ldots,e_{n-1}\right)$, the space of linear maps which coincide with $f$ on $e_0,e_1,\ldots,e_{n-1}$, it turns out that the $I$-adic topology on $\C[X]^{*}$ is coarser then the linear one. It follows then that $\C[X]^\bcirc\subseteq \C[X]^*$ is dense with respect to the $I$-adic topology as well and hence one may check that
\begin{equation*}
\prlimit{n}{\frac{\C[X]^\bcirc}{\C[X]^\bcirc\cap I^n}}\,\cong\, \prlimit{n}{\frac{\C[X]^*}{I^n}}\,\cong\, \C[X]^*.
\end{equation*}
\end{remark}

Now, consider the completion $\what{\uppsi}: \C[[\xi]] \to \what{\C[X]^{\circ}}$, where $\uppsi$ is the filtered monomorphism of algebras given in \eqref{Eq:SanSecondo}.  In view of Remark \ref{rem:Mucholio}, one shows that $\what{\iota} \circ \what{\uppsi} \, = \, \what{\vartheta}$. Therefore, $\what{\iota}$ is a split epimorphism, as $\what{\vartheta}$ is an homeomorphism whose inverse is $\Theta$.

\begin{remark} 
In fact  $\C[X]^*$ is a complete Hopf algebra in the sense of \cite[Appendix A]{quillen} and $\what{\iota}:\what{\C[X]^\bcirc}\to \C[X]^*$ becomes an effective epimorphism of complete Hopf algebras (see \cite[Proposition 2.19, page 274]{quillen}).
\end{remark}

The subsequent proposition gives conditions under which $\what{\iota}$ becomes an homeomorphism.  

\begin{proposition}\label{lema:losEmanueles}
The following assertions are equivalent
\begin{enumerate}
\item the canonical map $\what{\iota}:\what{\C[X]^\bcirc}\to \C[X]^*$ is injective,
\item the canonical map $\what{\iota}:\what{\C[X]^\bcirc}\to \C[X]^*$ is an homeomorphism,
\item the $J$-adic and the induced filtrations on $\C[X]^\bcirc$ coincide,
\item the $J$-adic and the induced topologies on $\C[X]^\bcirc$ are equivalent.
\end{enumerate}
\end{proposition}
\begin{proof}
We already observed that $\what{\uppsi} \circ \Theta$ is a continuous section of $\what{\iota}$. Thus, if  $\what{\iota}$ injective then it will be bijective with inverse $\what{\uppsi} \circ \Theta$, and so an homeomorphism. This proves the implication $(1)\Rightarrow (2)$.

To show that $(2)\Rightarrow(3)$, let us denote by $F_{n}\left(\what{\C[X]^\bcirc}\right) =\ker{ \what{\C[X]^\bcirc} \to \C[X]^\bcirc/J^{n} }$ the canonical filtration on $\what{\C[X]^\bcirc}$. If $\what{\iota}$ is an homeomorphism, then its inverse has to be $\what{\uppsi} \circ \Theta$. As a consequence, we obtain the second of the following chain of isomorphisms 
$$
\frac{\C[X]^\bcirc}{J^{n}} \, \cong\, \frac{\what{\C[X]^\bcirc}}{F_{n}\left(\what{\C[X]^\bcirc}\right)}  \,\cong \, \frac{\C[X]^*}{I^n}, 
$$
for every $n \geq 1$. Their composition sends $ p(x) + J^{n} \in \C[X]^\bcirc/ J^{n}$ to $\iota(p(x)) + I^{n} \in \C[X]^*/I^n$, which shows that $J^{n}\, =\, I^{n} \cap \C[X]^\bcirc$. Thus the two filtrations coincide. 

Since the implication $(3)\Rightarrow (4)$ is clear, let us show that $(4)\Rightarrow (1)$. Saying that the two topologies are equivalent, implies that every $J^n$ (which is open in the $J$-adic topology) has to be open in the induced topology as well. In particular, it has to contain an element of the neighbourhood base of $0$. Therefore, we may assume that for every $n \geq 0$, there exists $m \geq n$ such that $I^{m}\cap \C[X]^\bcirc \, \subseteq \, J^{n}$. Given $(f_{n}+ J^{n})_{n \geq 0}$ an element in the kernel of $\what{\iota}$, we have that $f_{n} \in I^{n}$ for every $n \geq0$.  This implies that for every $n \geq 0$, there exists $m \geq n$ such that 
$$
f_{n}+ J^{n}\, =\, f_{m} + J^{n}  \,\in\,  \left(I^{m} \cap \C[X]^\bcirc\right) + J^{n} \,= \, J^{n},
$$
which means that $(f_{n}+ J^{n})_{n \geq 0}\,=\, 0$ and this settles the proof.
\end{proof}

In conclusion, it follows from the result of \S\ref{sec:equivfiltrations} that none of the equivalent conditions in Proposition \ref{lema:losEmanueles} holds, as the two filtrations do not coincide. An explicit non-zero element which lies in the kernel of $\what{\iota}$ is exactly the one coming from equation \eqref{Eq:CarloAlberto}. Indeed, on the one hand
\begin{equation*}
\left(\phi_1-\sum_{k=0}^n\frac{1}{k!}\xi^k+J^{n+1}\right)_{n\geq 0} \,\in\, \what{\C[X]^\bcirc}
\end{equation*}
is non-zero, but on the other hand a direct check shows that
\begin{equation*}
\what{\iota}\left(\left(\phi_1-\sum_{k=0}^n\frac{1}{k!}\xi^k+J^{n+1}\right)_{n\geq 0}\right)=\left(\phi_1-\sum_{k=0}^n\frac{1}{k!}\xi^k+I^{n+1}\right)_{n\geq 0}=0
\end{equation*}
in $\C[X]^*$.

\begin{remark}
Observe that an element $\left(f_n+J^{n+1}\right)_{n\geq0}$ in $\what{\C[X]^\bcirc}$ can be considered as the formal limit $\limn \left(f_n\right)$ of the Cauchy sequence $\left\{f_n\mid n\geq 0\right\}$ in $\C[X]^\bcirc$ with the $J$-adic topology. The element $\left(\phi_1+J^{n+1}\right)_{n\geq 0}$ can be identified with $\phi_1$ itself, as limit of a constant sequence. On the other hand, the element $\left(\sum_{k=0}^n\xi^k/k!+J^{n+1}\right)_{n\geq 0}$ can be considered as the limit $\limn \left(\sum_{k=0}^n\xi^k/k!\right)$. As we already noticed, $\phi_1$ is associated with the exponential function, in the sense that its power series expansion in $\C[X]^*$ is $\sum_{k\geq 0}\xi^k/k!=\mathsf{exp}(\xi)$. However, it follows from what we showed that in $\what{\C[X]^\bcirc}$ the Cauchy sequence $\left\{\sum_{k=0}^n\xi^k/k!\mid n\geq 0\right\}$ does not converge to $\phi_1$.
\end{remark}

\section{Final remarks}\label{sec:finalremarks}

As we mentioned in the introduction, linearly recursive sequences have already been studied deeply as ``rational'' power series. What we plan to do in this section is to provide a possible explanation of why the topological richness expounded in the previous sections  didn't enter the picture before and to provide an overview of the different interpretations of these sequences.

The commutative diagram of algebras in \eqref{Eq:PornComb} summarizes the state of the art. Therein, $\C^{\N}$ is endowed with the algebra structure given by the product $\Sqx \, \Sqy =\big(\sum_{k=0}^{n} x_{k}y_{n-k}\big)_{\Sscript{n \geq 0}}$.
\begin{equation}\label{Eq:PornComb}
\begin{gathered}
\xymatrix@R=15pt{   \cL in(\C)  \ar@{_(->}[rd]  \ar@{->}^-{\cong}[rrr] & & & \C[X]^{\circ}   \ar@{_(->}^{\iota}[rd]  &  \\ & \hH\C^{\N}  \ar@{->}^{\Upphi}[rrr]   \ar@{->}^-{\cong}_-{\zeta}[dd] & & & \C[X]^{*}  \ar@{->}_-{\cong}^{\Theta}[dd]  \\ & & & &  \\  & \C^{\N}  \ar@{->}^{\Omega}[rrr] & & & \C[[Z]]  \\   \cL in(\C) \ar@{^(->}[ru]   \ar@{->}_-{\cong}^{\omega}[rrr] & & & \C[Z]_{\langle Z \rangle}  \ar@{^(->}[ru]  &  }
\end{gathered}
\end{equation}
The isomorphism $\Omega$ sends any sequence $\Sqx$ to the power series $\sum_{\Sscript{n\, \geq 0}}x_{n}Z^{n} $, while the isomorphism $\zeta$ sends a sequence $\Sqz$ to the sequence $(z_{n}/n!)_{\Sscript{n \, \geq 0}}$. The algebra $ \C[Z]_{\langle Z \rangle} $ denotes the localization of $\C[Z]$ at the maximal ideal $\langle Z \rangle$, that is, the set of fractions $p(Z)/q(Z)$ with $q(0) \neq 0$. Lastly, the isomorphism  $\omega$ is induced by  the restriction of $\Omega$ to $\cL in (\C)$ and it is given as follows. For a sequence $\Sqa$ in $\cL in (\C)$, let  $c_{r}=1, c_{r-1},\ldots,c_0  \in\C$, $r\geq 1$ be the family of constant coefficients such that
\begin{equation*}
a_{l+r} + c_{r-1}a_{l+r-1} + c_{r-2} a_{l+r-2} + \cdots + c_0 a_{l}\, =\, 0, \qquad \text{for all } l\geq 0.
\end{equation*}
If we consider the polynomials $q(Z)=\sum_{i=0}^{r} c_{r-i}Z^{i}$ and $p(Z)=\sum_{j=0}^{r-1}\big( \sum_{i=0}^{j} c_{r-i} a_{j-i}\big) Z^{j}$, then $q(Z)\big(\sum_{\Sscript{n\, \geq 0}}a_{n}Z^{n} \big)\, =\, p(Z)$. Thus, $\omega$ acts via $\omega(\Sqa) := p(Z)/q(Z) \, \in \C[Z]_{\langle Z \rangle}$. 

As one can realize from diagram \eqref{Eq:PornComb}, there are essentially two linear topologies which can be induced on $\cL in (\C)$: one from  $\hH\C^{\N}$, which we denote by $\tT_{\Sscript{\hH}}$, and the other from $\C^{\N}$, which we denote by $\tT$. Apart from these, $\cL in (\C)$  has its own two adic topologies given by the ideals  $\cI:=\cL in(\C) \cap \mathfrak{a}$, where $\mathfrak{a}$ is the augmentation ideal of  $\hH\C^{\N}$, and  $\cJ:= \cL in(\C) \cap \mathfrak{b}$, where $\mathfrak{b}$ is the augmentation ideal of  $\C^{\N}$.

It follows from the definitions that the $\cI$-adic topology on $\cL in(\C)$ is finer than $\tT_{\Sscript{\hH}}$ and the $\cJ$-adic one is finer than $\tT$. On the one hand, in view of the previous sections, the $\cI$-adic topology is in fact strictly finer than $\tT_{\Sscript{\hH}}$. On the other hand, however, one can show that the $\cJ$-adic topology turns out to be equivalent to $\tT$, since it is known that $\what{\C[Z]_{\langle Z \rangle}}$ is homeomorphic to $\what{\C[Z]}\cong\C[[Z]]$, and this may be the reason why topologies on $\cL in(\C)$ weren't analysed before.

Finally, comparing the topologies $\tT$ and  $\tT_{\Sscript{\hH}}$ on $\cL in (\C)$ seems to be more involved. Apparently it is possible that these are different. However, it is not clear to us how to show, for instance, that any open neighbourhood of the form $\cL in (\C) \cap \fk{a}^{n}$ (the product is that of $\C^{\N}$) is not contained in some open $\cL in (\C) \cap \fk{b}^{m}$ (the product now is in $\hH\C^{\N}$). What is clear instead is that the isomorphism $\zeta$ does not map linearly recursive sequences in $\hH\C^{\N}$ to linearly recursive sequences in $\C^{\N}$.

\bigskip

\noindent\textbf{Aknowledgements:} The authors thank A. Ardizzoni for the fruitful (coffees and) discussions. The first author would like to thank all the members of the dipartimento di Matematica ``Giuseppe Peano'' for providing a very fruitful working ambience.

\end{document}